\documentclass[a4paper,12pt, reqno]{amsart}

\usepackage{amsfonts, amsthm, amsmath, amssymb}

\usepackage{a4wide}

 \newtheorem{theorem}{Theorem}
 \newtheorem*{lemma}{Lemma}
 \newtheorem*{corollary}{Corollary}
 \theoremstyle{definition}
\newtheorem*{TA}{Acknowledgements}
 \newcommand{\mc}{\mathcal}

 \newcommand{\D}{\mc{D}}

 \newcommand{\mbb}{\mathbb}
 
 \newcommand{\N}{\mbb{N}}
 
 \newcommand{\Z}{\mbb{Z}}

\RequirePackage{mathrsfs} \let\mathcal\mathscr

\renewcommand{\phi}{\varphi}
\renewcommand{\rho}{\varrho}

\newcommand{\ZZ}{\mathbb{Z}}

\renewcommand{\leq}{\leqslant}
\renewcommand{\le}{\leqslant}
\renewcommand{\geq}{\geqslant}
\renewcommand{\ge}{\geqslant}

\newcommand{\m}{\mathbf{m}}

\newcommand{\ve}{\varepsilon}

\DeclareMathOperator{\Mod}{mod}
\renewcommand{\bmod}[1]{\,(\Mod{#1})}
\renewcommand{\=}{\equiv}

\begin{document}

\title[Mean value theorem and multiplicative inverses]{Incomplete Kloosterman sums and\\ multiplicative inverses in short intervals}


\author{T. D. Browning}

\author{A. Haynes}

\address{School of Mathematics, University of Bristol, Bristol, BS8 1TW, UK}

\email{t.d.browning@bristol.ac.uk}
\email{alan.haynes@bristol.ac.uk}


\begin{abstract}
We investigate the solubility of the congruence $xy\=1\bmod p$, where $p$ is a prime and $x, y$ are restricted to lie in suitable  short intervals.
Our work relies on a mean value theorem for incomplete Kloosterman sums.
\end{abstract}

\maketitle

\section{Introduction}
Let $p$ be a prime, and let $I_1, I_2\subseteq (0,p)$ be subintervals. This paper is motivated by determining conditions on $I_1,I_2$ under which we can ensure the solubility of the  congruence
\[
xy\=1\bmod p,\quad (x,y)\in I_1\times I_2.
\]
From a heuristic point of view we would expect this congruence to have a solution whenever $|I_1|,|I_2|\gg p^{1/2}$. However,  as highlighted by Heath-Brown \cite{H-B2000}, the best result to date requires that
$|I_1|\cdot|I_2|\gg p^{3/2}\log^2 p$. The  proof requires one to  estimate incomplete Kloosterman sums
\[
S(n,H)=\sum_{\substack{m=n+1\\m\not\= 0\bmod p}}^{n+H}e\left(\frac{\ell\overline{m}}{p}\right),
\]
for $\ell\in (\ZZ/\ell\ZZ)^*$, for which the Weil bound yields
\begin{equation}\label{eq:weil}
|S(n,H)|\le 2(1+\log p) p^{1/2}.
\end{equation}
It has been conjectured by Hooley \cite{hooley} that
$S(n,H)\ll H^{1/2}q^\ve,$ for any $\ve>0$,
which would enable one to handle
intervals with $|I_1|,|I_2|\gg p^{2/3+\ve}.$
However such a bound appears to remain a distant prospect.

A different approach to this problem involves  considering a sequence of pairs of intervals $I_1^{(j)}, I_2^{(j)}$, for $1\le j\le J$, and to ask whether  there is a value of $j$ for which there is a solution to the congruence
\begin{equation}\label{eqn:intro.1}
xy\=1\bmod p,\quad (x,y)\in I_1^{(j)}\times  I_2^{(j)}.
\end{equation}
There are some obvious degenerate cases here. For example, if we suppose that $I_1^{(j)}=I_2^{(j)}$ for all $j$, and that these run over all intervals of a given length $H$, then we are merely asking whether there is positive integer $h\le H$ with the property that the congruence $x(x+h)\=1\bmod p$ has a solution $x\in\Z$. This
is equivalent to deciding  whether  the set
$\{h^2+4:1\le h\le H\}$
contains a quadratic residue modulo $p$. When $H=2$, therefore, it is clear that this problem has a solution for all primes $p=\pm 1\bmod 8$. We avoid considerations of this sort by assuming that at least one of our sequences of intervals is pairwise disjoint. The following is our main result.

\begin{theorem}\label{thm:inverses1}
Let $H,K>0$ and let $I_1^{(j)}, I_2^{(j)}\subseteq (0,p)$ be subintervals, for
$1\le j\le J$, such that
\[|I_1^{(j)}|=H\quad\text{and}\quad |I_2^{(j)}|=K\]
and
\[I_1^{(j)}\cap I_1^{(k)}=\emptyset\quad\text{for all}\quad j\not= k.\]
Then
there exists $j\in \{1,\ldots,J\}$ for which \eqref{eqn:intro.1} has a solution
if
\[
J\gg \frac{p^{3}\log^4 p}{H^2K^2}.
\]
\end{theorem}

If we take $J=1$ in the theorem then we retrieve the above result that \eqref{eqn:intro.1} is soluble when
$HK\gg p^{3/2}\log^2 p$. Alternatively, if we allow a larger value of $J$, then we can get closer to what would follow on
Hooley's hypothesised bound for $S(n,H)$.

\begin{corollary}
With notation as in Theorem \ref{thm:inverses1}, suppose that $J\gg p^{1/3}$.
Then
there exists $j\in \{1,\ldots,J\}$ for which \eqref{eqn:intro.1} has a solution provided that
$H>p^{2/3}$ and $K>p^{2/3}(\log p)^2$.
\end{corollary}

Our proof of Theorem \ref{thm:inverses1}  relies upon a mean value estimate for incomplete Kloosterman sums. These types of estimates have been studied extensively for multiplicative characters, especially in connection with variants of Burgess's bounds (see Heath-Brown \cite{H-B2012} and the discussion therein).
The situation for Kloosterman sums is relatively under-developed (see Friedlander and Iwaniec \cite{FI}, for example). The result we present here appears to be new, although many of our techniques are borrowed directly from the treatment of the analogous multiplicative problem \cite[Theorem 2]{H-B2012}. The deepest part of our argument is an appeal to Weil's bound for Kloosterman sums.  We will prove the following result in the next section.

\begin{theorem}\label{thm:mvt1}
If $I_1,\ldots ,I_J\subseteq (0,p)$ are disjoint subintervals,  with $H/2<|I_j|\leq H$ for each $j$, then for any $\ell\in (\ZZ/p\ZZ)^*$,  we have
\[
\sum_{j=1}^J\left|\sum_{n\in I_j}e\left(\frac{\ell\overline{n}}{p}\right)\right|^2\leq 2^{12} p\log^2H.
\]
\end{theorem}

Taking $J=1$ shows that, up to a constant factor,  this result includes as a special case the bound \eqref{eq:weil} for incomplete Kloosterman sums.

\begin{TA}
While working on this paper TDB was supported by
EPSRC grant  \texttt{EP/E053262/1} and
AH was supported by EPSRC grant \texttt{EP/J00149X/1}.
We are grateful to  Professor Heath-Brown for several useful discussions.
\end{TA}

\section{Proof of Theorem \ref{thm:mvt1}}

Our starting point is the following mean value theorem for $S(n,H)$.

\begin{lemma}\label{lem:mvt1}
For $H\in\N$ and $\ell\in (\ZZ/p\ZZ)^*$, we have
$$
\sum_{n=1}^p\left|S(n,H)\right|^2\leq \frac{H^2}{p}+8pH.
$$
\end{lemma}
\begin{proof}
After squaring out the inner sum and interchanging the order of summation, the left hand side becomes
$$
\sum_{h_1,h_2=1}^H\sum_{\substack{n=1\\n\not= -h_1,-h_2\bmod p}}^pe\left(\frac{\ell(\overline{n+h_1}-\overline{n+h_2})}{p}\right).
$$
Using orthogonality of characters it is easy to see that the inner sum over $n$ is
\begin{align*}
&=\sum_{\substack{x,y\in (\ZZ/p\ZZ)^*\\ \overline{x}-\overline{y}\=h_1-h_2\bmod p}}
e\left(\frac{\ell(x-y)}{p}\right)\\
&=
\frac{1}{p}
\sum_{\substack{x,y\in (\ZZ/p\ZZ)^*}}
e\left(\frac{\ell(x-y)}{p}\right)
\sum_{a=1}^p
e\left(\frac{a(  \overline{x}-\overline{y} )+a(h_2-h_1)}{p}\right).
\end{align*}
Hence
\begin{align*}
\sum_{n=1}^p\left|S(n,H)\right|^2
&=
\frac{1}{p}
\sum_{a=1}^p
\sum_{h_1,h_2=1}^H
e\left(\frac{a(h_2-h_1)}{p}\right)
\sum_{\substack{x,y\in (\ZZ/p\ZZ)^*}}
e\left(\frac{\ell(x-y)+a(  \overline{x}-\overline{y} )}{p}\right)\\
&=
\frac{1}{p}
\sum_{a=1}^p
|K(\ell,a;p)|^2
\sum_{h_1,h_2=1}^H
e\left(\frac{a(h_2-h_1)}{p}\right),
\end{align*}
where $K(\ell,a;p)$ is the usual complete Kloosterman sum. The contribution from $a=p$ is
$$
\frac{|K(\ell,0;p)|^2 H^2}{p}=\frac{H^2}{p},
$$
since $p\nmid \ell$. The remaining contribution has modulus
\begin{align*}
\left|\frac{1}{p}
\sum_{a=1}^{p-1}
|K(\ell,a;p)|^2
\sum_{h_1,h_2=1}^H
e\left(\frac{a(h_2-h_1)}{p}\right)\right|
&\leq
4
\sum_{0<|a|\leq p/2}
\left|\sum_{h_1,h_2=1}^H
e\left(\frac{a(h_2-h_1)}{p}\right)\right|\\
&\leq
8
\sum_{0<a\leq p/2}
\min\left\{H, \frac{p}{2a}\right\}^2\\
&\leq 8 pH,
\end{align*}
by the Weil bound for the Kloosterman sum and the familiar estimate for a geometric series.
Combining these contributions, we therefore arrive at the statement of the lemma.
\end{proof}

The rest of the proof of Theorem \ref{thm:mvt1} is taken from the proof of \cite[Theorem 2]{H-B2012}, and we include
it only for completeness. We may assume that $H\geq 4$ in what follows since the result is trivial otherwise. Write $N_j$ for the smallest integer in $I_j$ and suppose that $N_1<\cdots <N_J$. By separately considering the odd and then the even numbered intervals we may assume without loss of generality that $N_{j+1}-N_j\ge H$ for $1\le j <J$.

The starting point is the observation that
\begin{align}\label{eqn:pf:thm:mvt1.1}
\sum_{j=1}^J\left|\sum_{n\in I_j}e\left(\frac{\ell\overline{n}}{p}\right)\right|^2\le\sum_{j=1}^J\max_{1\le h\le H}|S(N_j,h)|^2.
\end{align}
For any $1\le h\le H$ and $N_j-H<n\le N_j$ we have that
\[|S(N_j,h)|=|S(n,N_j-n+h)-S(n,N_j-n)|\le 2\max_{1\le k\le 2H}|S(n,k)|,\]
whence
\begin{align*}
|S(N_j,h)|\le\frac{2}{H}\sum_{N_j-H<n\le N_j}\max_{1\le k\le 2H}|S(n,k)|.
\end{align*}
Cauchy's inequality yields
\[|S(N_j,h)|^2\le\frac{4}{H}\sum_{N_j-H<n\le N_j}\max_{1\le k\le 2H}|S(n,k)|^2.\]
Taking the max over $h$ and then summing over $j$ now gives
\begin{equation}
\label{eqn:pf:thm:mvt1.2}
\begin{split}
\sum_{j=1}^J\max_{1\le h\le H}|S(N_j,h)|^2&\le\frac{4}{H}\sum_{j=1}^J\sum_{N_j-H<n\le N_j}\max_{1\le k\le 2H}|S(n,k)|^2\\
&\le\frac{4}{H}\sum_{n=1}^p\max_{1\le k\le 2H}|S(n,k)|^2,
\end{split}\end{equation}
the last inequality coming from our spacing assumption. We now seek an upper bound for the sum on the right hand side.

Let $t$ be the smallest positive integer with $2H\le 2^t$, so that in particular $2H\leq 2^t\leq 4H$ and
$t+1\leq 4\log H$. For each $1\le n\le p$ we choose a positive integer $k=k(n)\le 2H$, with
\[\max_{1\le h\le 2H}|S(n,h)|=|S(n,k)|.\]
By writing
$k=\sum_{d\in\D}2^{t-d},$
where $\D$ is a collection of integers in $[0,t]$, we have that
\[S(n,k)=\sum_{d\in\D}S(n+v_{n,d}2^{t-d},2^{t-d}),\]
where
\[v_{n,d}=\sum_{\substack{e\in\D\\ e<d}}2^{d-e}< 2^d.\]
Then by Cauchy's inequality we deduce that
\begin{align*}
\max_{1\le h\le 2H}|S(n,h)|^2&\le |\D|\sum_{d\in\D}|S(n+v_{n,d}2^{t-d},2^{t-d})|^2\\
&\le (t+1)\sum_{0\le d\le t}\sum_{0\le v<2^d}|S(n+v2^{t-d},2^{t-d})|^2.
\end{align*}
Now summing both sides over $n$ and applying Lemma \ref{lem:mvt1} we have that
\begin{align*}
\sum_{n=1}^p\max_{1\le k\le 2H}|S(n,k)|^2
&\leq
(t+1) \sum_{0\le d\le t}\sum_{0\le v<2^d} \left(\frac{2^{2t-2d}}{p}+8p2^{t-d}\right)\\
&\leq
(t+1)^2 \left(\frac{2^{2t}}{p}+8p2^{t}\right)\\
&\leq 2^{8}
 \left(\frac{H^2}{p}+2pH\right)\log^2H.
 \end{align*}
Since $H^2/p\leq 2pH$ we
easily complete the proof of Theorem \ref{thm:mvt1} by combining this  with (\ref{eqn:pf:thm:mvt1.1}) and (\ref{eqn:pf:thm:mvt1.2}).

\section{Proof of Theorem \ref{thm:inverses1}}

Now we proceed to the proof of our main theorem. For each $j$ the number of solutions to
\eqref{eqn:intro.1} is equal to
\begin{align*}
\sum_{x\in I_1^{(j)},~\overline{x}\in I_2^{(j)}}1
&=\sum_{x\in I_1^{(j)}}\sum_{y\in I_2^{(j)}}\frac{1}{p}\sum_{\ell=1}^pe\left(\frac{\ell(\overline{x}-y)}{p}\right)\\
&=\sum_{x\in I_1^{(j)}}\sum_{y\in I_2^{(j)}}\frac{1}{p}+\sum_{x\in I_1^{(j)}}\sum_{y\in I_2^{(j)}}\frac{1}{p}\sum_{\ell=1}^{p-1}e\left(\frac{\ell(\overline{x}-y)}{p}\right)=S_{1,j}+S_{2,j},
\end{align*}
say.
The total contribution from the $S_{1,j}$ terms is
\begin{equation}
\sum_{j=1}^JS_{1,j}\gg \frac{JHK}{p}.\label{eqn:pf:thm:inverses1.1}
\end{equation}
Next, the standard estimate for a geometric series gives
\begin{align*}
S_{2,j}
&=\frac{1}{p}\sum_{0<|\ell|\leq p/2}
\left(\sum_{y\in I_2^{(j)}}e\left(\frac{-\ell y}{p}\right)\right)\left(\sum_{x\in I_1^{(j)}}e\left(\frac{\ell\overline{x}}{p}\right)\right)\\
&\ll
\sum_{0<|\ell|\leq p/2}
\frac{1}{|\ell|}\left|\sum_{x\in I_1^{(j)}}e\left(\frac{\ell\overline{x}}{p}\right)\right|.
\end{align*}
Applying Cauchy's inequality and Theorem \ref{thm:mvt1} we deduce that
\begin{align*}
\sum_{j=1}^J|S_{2,j}|
&\ll
\sum_{0<|\ell|\leq p/2}
\frac{1}{|\ell|}\sum_{j=1}^J\left|\sum_{x\in I_1^{(j)}}e\left(\frac{\ell\overline{x}}{p}\right)\right|\nonumber\\
&\ll J^{1/2}\sum_{0<|\ell|\leq p/2}
\frac{1}{|\ell|}\left(\sum_{j=1}^J\left|\sum_{x\in I_1^{(j)}}e\left(\frac{\ell\overline{x}}{p}\right)\right|^2\right)^{1/2}\nonumber\\
&\ll J^{1/2}(\log p)\left(p\log^2 H\right)^{1/2}.
\end{align*}
Under the conditions of Theorem \ref{thm:inverses1}, it now follows that (\ref{eqn:pf:thm:inverses1.1}) dominates this quantity, from which the conclusion of the theorem follows.

\end{document}